\documentclass{amsart}

\pdfoutput=1

\usepackage{yhmath}

\newcommand{\Rr}{\mathcal{R}}

\newcommand{\Ii}{\mathcal{I}}

\newcommand{\spann}{\mathrm{span}}
\newcommand{\Sym}{\mathrm{Sym}}

\newcommand{\End}{\mathrm{End}}

\newcommand{\Ort}[1]{\mathrm{O}(#1)}

\newcommand{\GL}[1]{\mathrm{GL}({#1})}

\newcommand{\R}{\mathbb{R}}
\newcommand{\Z}{\mathbb{Z}}
\newcommand{\N}{\mathbb{N}}

\newcommand{\CC}{\mathbb{C}}

\newcommand{\RR}{\mathbb{R}}

\usepackage{stmaryrd}
\usepackage{tikz}

\usepackage{amssymb}
\usepackage{microtype}
\usepackage{pgfplots}
\pgfplotsset{compat=1.18} 
\usetikzlibrary{decorations.pathreplacing}
\usepackage[pdftex,colorlinks,citecolor=black,linkcolor=black,urlcolor=black,bookmarks=false]{hyperref}
\def\MR#1{\href{http://www.ams.org/mathscinet-getitem?mr=#1}{MR#1}}
\def\arXiv#1{arXiv:\href{http://arXiv.org/abs/#1}{#1}}
\usepackage{doi}
\usepackage{booktabs}
\usepackage[c2]{optidef}
\usepackage{ytableau}
\usepackage{algorithm}
\usepackage{algpseudocode}
\usepackage[tableposition=above]{caption}

\newtheorem{theorem}{Theorem}[section]
\newtheorem{proposition}[theorem]{Proposition}

\newtheorem{lemma}[theorem]{Lemma}
\theoremstyle{remark}

\theoremstyle{definition}

\numberwithin{equation}{section}
\numberwithin{table}{section}
\numberwithin{figure}{section}
\title{Optimality and uniqueness of the $D_4$ root system}

\author{David de Laat}
\address{D.\ de Laat, Delft Institute of Applied Mathematics\\
Delft University of Technology\\ Delft, The Netherlands} \email{d.delaat@tudelft.nl}

\author{Nando Leijenhorst}
\address{N.\ M.\ Leijenhorst, Delft Institute of Applied Mathematics\\
Delft University of Technology\\ Delft, The Netherlands} \email{n.m.leijenhorst@tudelft.nl}

\author{Willem de Muinck Keizer}
\address{W.\ H.\ H.\ de Muinck Keizer, Delft Institute of Applied Mathematics\\
Delft University of Technology\\ Delft, The Netherlands} \email{w.h.h.demuinckkeizer@tudelft.nl}

\date{May 27, 2024} 

\begin{document}

 \begin{abstract}
     We prove that the $D_4$ root system (the set of vertices of the regular $24$-cell) is the unique optimal kissing configuration in $\R^4$, and is an optimal spherical code. For this, we use semidefinite programming to compute an exact optimal solution to the second level of the Lasserre hierarchy. We also improve the upper bound for the kissing number problem in $\R^6$ to $77$.
 \end{abstract}


\maketitle

\tableofcontents

\section{Introduction}

A kissing configuration in dimension $n$ is a collection of nonoverlapping, equal-size spheres in $\R^n$ that touch (or ``kiss'') a central sphere of the same size. We will assume the spheres have unit radius, and we identify a kissing configuration with the set $C$ of contact points with the central sphere. Such a set $C$ is a spherical code with minimal angle at least $\pi/3$. The kissing number $k(n)$ in dimension $n$ is the maximum size of such a set $C$.

In dimension four, a kissing configuration is given by the $D_4$ root system.
This root system may be constructed as the set of all $24$ vectors in $\R^4$ with integer coordinates and length $\sqrt{2}$, referred to as the roots. In this paper we normalize the roots to have unit length.
Viewed geometrically, the roots form the vertices of the $24$-cell, which is one of the six regular polytopes in dimension four. 
The possible inner products between distinct roots are $0$, $\pm 1/2$, and $-1$.
Hence $D_4$ is a kissing configuration in dimension 4 of size $24$; that is, $k(4) \geq 24$. 
In 2008, Musin showed $k(4) = 24$; the $D_4$ root system is an optimal kissing configuration \cite{musin08}.

In this paper, we show it is unique.
More precisely, we show that the $D_4$ root system is the only optimal kissing configuration in dimension four up to isometry. This implies it satisfies the stronger geometric condition of being an optimal spherical code: it minimizes 
\[
t_{\mathrm{max}}(C) = \max_{\substack{x, y \in C \\ x \neq y}} \langle x, y \rangle
\]
over all sets $C$ consisting of $24$ points in the unit sphere $S^3 = \{ x \in \R^4 \mid \langle x, x \rangle = 1\}$. This is in contrast with the result by Cohn, Conway, Elkies, and Kumar \cite{MR2367321} that the $D_4$ root system is not universally optimal, meaning that there exists an absolutely monotonic function $f$ (a smooth function with all derivatives nonnegative on $[-1,1]$) for which $D_4$ does not minimize
\[
\sum_{\substack{x, y \in C \\ x \neq y}} f(\langle x, y \rangle)
\]
over all $C \subseteq S^3$ of size $24$. In their paper, they conjecture that no universally optimal spherical code of $24$ points exists in $S^3$. The combination of the $D_4$ root system being the unique optimal spherical code, but not a universally optimal spherical code, proves this conjecture. 

The kissing number problem has a rich history, going back to a discussion between Newton
and Gregory in 1694 on the correct value of $k(3)$, which was resolved in 1953 by Sch\"utte and Van der Waerden \cite{MR0053537}. Currently, the value of $k(n)$ is known for $n=1$, $2$, $3$, $4$, $8$, and $24$. For background on the kissing number problem, we refer to \cite{pfender04}.  

In 1973, Delsarte introduced the linear programming bound, which can be used to bound the sizes of codes over finite alphabets \cite{MR0384310}. Delsarte, Goethals, and Seidel adapted this to the sphere so that it can be used to compute upper bounds on $k(n)$ \cite{delsarte77}.  Remarkably, this bound is sharp in dimensions $8$ and $24$, where by a sharp bound we mean that the optimal objective value is exactly equal to the kissing number, without having to take the integer part. The optimal objective value is $240$ in dimension $8$ and $196560$ in dimension $24$, which coincides with the sizes of the kissing configurations obtained by taking the shortest nonzero vectors in the $E_8$ root lattice and the Leech lattice $\Lambda_{24}$ \cite{odlyzko79, MR0529659}. This proves optimality of those configurations, and since the bound is sharp, complementary slackness holds, which was used to prove uniqueness \cite{Bannai_Sloane_1981}. 

In dimension four, the Delsarte bound was used to show $k(4) \leq 25$, which was the first improvement over Coxeter's upper bound of $26$ from 1964 \cite{MR0530296,MR0164283}. In \cite{ArestovBabenko} it was shown that the Delsarte bound cannot be used to prove $k(4) = 24$, and Musin's optimality proof for the $D_4$ root system uses a strengthening of the Delsarte bound. However, this strengthening does not lead to a sharp bound.

The Delsarte bound is called a two-point bound since it considers constraints between pairs of points on the sphere. Bachoc and Vallentin developed the three-point semidefinite programming bound for spherical codes, adapted from Schrijver's three-point bound for binary codes  \cite{bachoc08,MR2236252}. The three-point bound recovers the optimality results in dimensions $3$ and $4$ and improves the best-known upper bound for the kissing number problem in many other dimensions. To compute the three-point bound it is first reduced to a finite-dimensional problem by truncating an inverse Fourier transform, and since its introduction in 2008, all improvements to upper bounds on $k(n)$ have come from increasing this truncation degree \cite{MR2676746,machado16,leijenhorst2023solving}. 

The numerical data (see \cite[Table 6.1]{leijenhorst2023solving} for the newest results), however, suggests that the three-point bound for the kissing number problem is not sharp in any dimension $3 \leq n \leq 24$ and any truncation degree, except for the cases $n=8$ and $n=24$ where the Delsarte bound is already sharp.
There has been considerable work on $k$-point bound generalizations of the three-point bound, but this has not yet resulted in sharp, or even improved, bounds for the kissing number problem (or spherical code problems in general) \cite{musin14,44bd101c840a4621baf53c7d8a6a5f9b,Musin2018,deLaat2021,musin2023semidefinite,bekker2023convergence}.

Over the last decades, the moment/sums-of-squares approach by Lasserre and Parrilo (see \cite{lasserre01,lasserre02,Par00}) has become an important tool in mathematical optimization and theoretical computer science. Applying the Lasserre hierarchy to the independent set problem in a finite graph gives a converging hierarchy of increasingly large semidefinite programs giving successively stronger upper bounds on the independence number. We can think of the kissing number problem as the independent set problem in the graph on the unit sphere $S^{n-1}$, where two distinct vertices $x, y \in S^{n-1}$ are adjacent if $\langle x, y\rangle  > 1/2$. In \cite{laat15}, De Laat and Vallentin generalized this hierarchy to infinite graphs such as these, giving a hierarchy of $2t$-point bounds, where $t$ is the level of the hierarchy. In principle, this solves the kissing number problem in any dimension, since this hierarchy converges in finitely many steps. In practice, computing the levels of this hierarchy beyond the first level (which reduces to the Delsarte linear programming bound) is challenging.

In this paper, we compute the second level of this hierarchy for spherical code problems. We show the second level of the hierarchy is sharp for the kissing number problem in dimension four (the upper bound is exactly $24$) by computing an exact optimal solution. We then use complementary slackness to extract a uniqueness proof for the $D_4$ root system from the optimal solution. 

This is the first time the second level of the Lasserre hierarchy has been computed for a spherical code problem and the first improvement over the three-point bounds for spherical codes. 
Previously, the second level of the Lasserre hierarchy has been computed for two problems on infinite graphs. Firstly, it has been computed for energy minimization on the two-dimensional sphere \cite{laat16}. The techniques used there, however, become too expensive when going to higher dimensional spheres or further truncation degree of the inverse Fourier transform, and computing a sharp bound for the kissing number problem in dimension four would be prohibitively expensive with the techniques from that paper.

In \cite{delaat2023lasserre}, De Laat, Machado, and De Muinck Keizer compute the second and third levels of the hierarchy for the equiangular lines problem with a fixed angle $\theta$. Here the corresponding graph on $S^{n-1}$ has an edge between distinct points $x$ and $y$ if $\langle x, y \rangle \neq \pm \cos \theta$.
Although this is an infinite graph, the quotient space $\Ii_t / \Ort{n}$, where $\Ii_t$ is the set of independent sets of size at most $t$ and $\Ort{n}$ is the orthogonal group, is finite. Here $\Ort{n}$ acts on $\Ii_t$ by $g \{x_1,\dots, x_k\} = \{g x_1, \dots, g x_k\}$. 
For the kissing number problem, the quotient space $\Ii_t / \Ort{n}$ is infinite.
Because of this, computing the hierarchy for the kissing number problem is more involved, and in this paper, we extend the techniques from \cite{delaat2023lasserre} to do this.

The $t$-th level of the hierarchy is the optimization problem
\begin{equation}\label{pr:las}
\begin{aligned}
& \text{minimize} & & K(\emptyset, \emptyset)\\
& \text{subject to} & & K \in \mathcal C(\Ii_t \times \Ii_t)_{\succeq 0},\\
& & & A_t K(Q) \leq -1_{\Ii_{=1}}(Q), \, Q \in \Ii_{2t} \setminus \{\emptyset\}.
\end{aligned}
\end{equation}
Here $\mathcal C(\Ii_t \times \Ii_t)_{\succeq 0}$ is the cone of continuous, positive kernels on $\Ii_t$, where $\Ii_t$ inherits its topology from $S^{n-1}$ (see \cite{laat15}), $1_{\Ii_{=1}}$ is the indicator function of the set $\Ii_{=1}$ of one-element subsets of $S^{n-1}$, and 
\[
A_t K(Q) = \sum_{\substack{J_1,J_2 \in \Ii_t\\  J_1 \cup J_2 = Q}} K(J_1,J_2).
\]
Any feasible solution $K$ provides an upper bound on $k(n)$; see the start of the proof of Lemma~\ref{lem:innerproductsincode} for the argument. Moreover, if $K$ is feasible, then the kernel
\[
(J_1, J_2) \mapsto \int_{\Ort{n}} K(\gamma J_1, \gamma J_2) \, d\gamma
\]
is also feasible and has the same objective value, from which it follows we may restrict to $\Ort{n}$-invariant kernels.

To reduce this to a finite-dimensional problem, we express such an $\Ort{n}$-invariant kernel $K$ in terms of its inverse Fourier transform and truncate the series. For each $\lambda \in \Z^t$ with $\lambda_1 \ge \ldots \ge \lambda_t \ge 0$, we define a unitary representation $\pi \colon \Ort{n} \to V$ and denote the space of continuous, $\Ort{n}$-equivariant maps from $\Ii_t$ to $V$ by
$
\mathrm{Hom}_{\Ort{n}}(\Ii_t, V).
$
We refer to $|\lambda| = \sum_i \lambda_i$ as the degree of $\pi$.
We will construct a family $\{ \psi_{\lambda,\ell}\}$ of elements in this space and define the matrix $Z_\lambda(J_1,J_2)$ by
\[
Z_\lambda(J_1,J_2)_{\ell_1,\ell_2} = \langle \psi_{\lambda, \ell_1}(J_1), \psi_{\lambda,\ell_2}(J_2) \rangle,
\]
where the inner product on $V$ is used. For this, we need an explicit description of the space of invariants $V^{\Ort{n-t}}$. We use the construction by Gross and Kunze \cite{gross1984finite} of irreducible representations $V$ induced by representations of $\GL{t}$, which have nontrivial space of invariants $V^{\Ort{n-t}}$.

For each $\lambda$, let $\widehat K_\lambda$ be a positive semidefinite matrix of the same size as $Z_\lambda$ with only finitely many nonzero entries. Then the kernel $K \colon \Ii_t \times \Ii_t \to \RR$ defined by
\[
K(J_1, J_2) = \sum_{|\lambda| \le d} \langle \widehat K_\lambda,  Z_\lambda(J_1,J_2) \rangle
\]
is continuous, positive, and $\Ort{n}$-invariant. With the right choice of representations and families of equivariant functions, these approximate all continuous, positive, $\Ort{n}$-invariant kernels.
This last statement will not be discussed in this paper, since it is not necessary for the main result.

We have two main technical contributions. In \cite{delaat2023lasserre}, the zonal matrices $Z_\lambda$ are constructed for the equiangular lines case, where the set of independent sets has finitely many orbits and where there are only finitely many pointwise constraints. In Section~\ref{sec:zonal}, we extend this to infinitely many orbits for the case $t=2$, and we give a rescaling so that the entries of $Z_\lambda(J_1, J_2)$ become polynomials in the inner products between the vectors in $J_1 \cup J_2$. This allows us to reduce \eqref{pr:las} to a finite-dimensional problem by truncating the inverse Fourier transform, and to write the constraints using sums-of-squares characterizations, which means we can use semidefinite programming to compute bounds. 

Our second technical contribution concerns the computation of the zonal matrices for $t=2$. In our approach of generating the zonal matrices via representations of $\Ort{n}$ induced by representations of $\mathrm{GL}(t)$, we identify additional symmetries under certain actions of $\Ort{t}$ and $\Ort{n-t}$, and we use this to significantly reduce the computations that need to be performed; see Section~\ref{sec:speedups}. To obtain a sharp bound for the kissing number problem in $\R^4$ we need the zonal matrices $Z_\lambda$ with $|\lambda| \le 14$, and for this these reductions are essential.

Cohn and Elkies \cite{MR1973059} gave a noncompact adaptation of the Delsarte linear programming bound, and conjectured it gives the optimal sphere packing density in dimensions $8$ and $24$. Note that for noncompact problems, such as the sphere packing problem, one needs a sharp bound to prove optimality. In \cite{MR3664816}, Viazovska proved the groundbreaking result that the $E_8$ root lattice gives an optimal sphere packing in $\R^8$ by constructing an optimal solution to the Cohn-Elkies bound, after which optimality of the Leech lattice $\Lambda_{24}$ was shown similarly in \cite{MR3664817}. Currently, the sphere packing problem has been solved in dimensions $1$, $2$, $3$, $8$, and $24$, where the proof for the three-dimensional case used a completely different approach \cite{hales2005proof}.

It is conjectured that the $D_4$ lattice gives the optimal sphere packing in dimension four, where optimality among lattice packings has been known since 1873 \cite{MR1509795}. A numerically sharp three-point bound for the lattice sphere packing problem in $\R^4$ has recently been computed in \cite{cohnlaatsalmon}, but a (numerically) sharp bound for the general sphere packing problem is not known in dimension four. Since we show the second level of the Lasserre hierarchy is sharp for the kissing number problem in dimension four (just as the Delsarte bound is sharp in dimension $8$ and $24$), one might expect a noncompact adaptation (see also \cite{cohnsalmon}) might be sharp for the sphere packing problem in dimension four (as is the Cohn-Elkies bound in dimensions $8$ and $24$). We therefore suspect the second level of the Lasserre hierarchy gives a viable approach to solving the sphere packing problem in dimension four.

In this paper, we focus on the four-dimensional case. The reason for this is that computing the zonal matrices and solving the semidefinite programs is computationally expensive, and we have only performed computations for $|\lambda| \leq 16$. In the same way as for the three-point bounds, this is the truncation degree around which the bounds start to improve on the Delsarte bound. For the four-dimensional case of the kissing number problem, this results in a sharp bound, but it seems that in most other dimensions the degree is not yet high enough to get improved bounds. For the six-dimensional case, we report a small improvement in the upper bound from $78$ to $77$, which is the first improvement since the introduction of the three-point bound.

The paper is organized as follows. In Section~\ref{sec:zonal}, we construct a system of equivariant functions for which the corresponding zonal matrices consist of polynomials in the inner products. In Section~\ref{sec:speedups}, we show how these zonal matrices can be computed efficiently. In Section~\ref{sec:sdp} we discuss the semidefinite programming formulation, and in Section~\ref{sec:applications} we discuss the applications.

\section{Equivariant functions and zonal matrices}\label{sec:zonal}

\subsection{Representations of the general linear group}
\label{sec:gl2reps}

We start by briefly recalling some facts about the representations of the general linear group, which may be found, e.g., in \cite[Chapter 15]{fulton91}.
The irreducible representations of $\GL{t}$ are indexed by their signature $\lambda = (\lambda_1, \dots, \lambda_t)$, which is a tuple of integers satisfying $\lambda_1 \geq \lambda_2 \geq \ldots \geq \lambda_t$. The polynomial, irreducible representations are those with~$\lambda_t \geq 0$.

Since we consider the second level of the Lasserre hierarchy, we will require an explicit description of the irreducible, polynomial representations of $\GL{t}$ for $t=2$.
They are given by 
\[
W = \Sym^{\lambda_2}(\wedge^2 U) \otimes \Sym^{m}(U), 
\]
where $U = \CC^2$ is the tautological representation that sends a matrix to itself with basis $e_1, e_2$, the signature $\lambda = (\lambda_1, \lambda_2)$ satisfies $\lambda_1 \geq \lambda_2 \geq 0$, and we define $m = \lambda_1 - \lambda_2$.
We denote the corresponding group homomorphism by $\rho \colon \GL{2} \to \mathrm{GL} (W) $.  
A basis of this representation is given by 
\[
	w_k = (e_1 \wedge e_2)^{\lambda_2} e_1^{m-k} e_2^{k},
\]
where $k = 0, 1, \dots, m$.
We give $W$ the inner product such that $\langle w_{k_1}, w_{k_2} \rangle = \delta_{k_1,k_2} $. 
With this choice, we have
\begin{align}\label{eq:rhomatrixcoeffs}
    &\langle w_{k_1}, \rho(A) w_{k_2} \rangle\\
 &\quad= \det(A)^{\lambda_2} \sum_{l = 0}^{m - k_1} \tbinom{m - k_2}{l} \tbinom{k_2}{m - k_1 - l} A_{11}^l A_{21}^{m - k_2 - l} A_{12}^{m - k_1 - l} A_{22}^{k_2 - (m - k_1 - l)}.\nonumber
\end{align}
For brevity, we shall use the notation $\rho(A)_{k_1, k_2} = \langle w_{k_1}, \rho(A) w_{k_2} \rangle$.
Let $c_j(k)$ denote the number of times $e_j$ occurs in the tensor $w_k$. 
Concretely, we have $c_1(k) = \lambda_2 + m - k$ and $c_2(k) = \lambda_2 + k$.
For a diagonal matrix $D$, we have
\[
	\rho(D) w_k = D_{11}^{c_1(k)} D_{22}^{c_2(k)} w_k.
\]
We will occasionally refer to the representation as $\rho_{\lambda}$ when it is convenient to make the dependence on $\lambda$ explicit.

For later use, we also record here a formula for the differential $d\rho$ at the identity $I$ evaluated at
\[
X = \begin{bmatrix}
0 & 1 \\
-1 & 0
\end{bmatrix}.
\] 
Using the product rule (see, e.g., \cite[Chapter 8]{fulton91}), we obtain 
\[
    d\rho(X) w_{k_2} = -(m - k_2)w_{k_2 + 1} + k_2 w_{k_2 - 1}
\]
and hence $d \rho(X)_{k_1, k_2} = -(m - k_2)\delta_{k_1 , k_2 + 1} + k_2 \delta_{k_1 , k_2 - 1}$. 

\subsection{Invariants of the orthogonal group}
\label{sec:invariantso}

Let $n \geq 2t$.
Denote by $\Ort{n, K}$ the group of $n \times n$ matrices $g$ with entries in the field $K$  satisfying $g^{\sf T} g = I$.
We see the group $\Ort{n - t, K}$ as the subgroup of $\Ort{n, K}$ which fixes the first $t$ standard basis vectors.
We will denote $\Ort{n, \RR}$ by $\Ort{n}$.

Following Gross and Kunze \cite{gross1984finite}, we now define certain representations of $\Ort{n}$ induced by representations of $\GL{t}$.
Let $(\rho, W)$ be the polynomial, irreducible representation of $\GL{t}$ with signature $\lambda$.
Define the complex $t \times n$ matrix
\[
\omega = \begin{pmatrix} I_t & i I_t & 0 \end{pmatrix}
\]
and the $n \times t$ matrix
\[
\epsilon = \begin{pmatrix} I_t \\ 0 \end{pmatrix}.
\]
For each $w \in W$, define a function $f_w \colon \Ort{n, \CC} \to W$ by
\begin{equation} \label{eq:defoforthrep1}
    f_w (\gamma) = \rho(\omega \gamma \epsilon)w.
\end{equation}
Define the vector space of right translates of such functions by
\[
    V = \spann \left\{ R_{g} f_w \mid g \in \Ort{n,\CC},\, w \in W \right\},
\]
where $R_{g} f_w (\gamma) = f_w (\gamma g)$. 
This space is a representation of $\Ort{n, \CC}$ by right translation.
A representation of $\Ort{n}$ is obtained by restricting $\Ort{n, \CC}$ to $\Ort{n}$.
We shall refer to this representation of $\Ort{n}$ by $(\pi, V)$.

Let $\Psi \colon W \to V$ be the map sending $w$ to $f_w$,
and  consider the space of invariants
\[
V^{\Ort{n - t}} = \{ v \in V \, | \, \pi(h)v = v \textup{ for all } h \in \Ort{n - t} \}.
\]
Since $h \epsilon = \epsilon$ for $h \in \Ort{n - t}$, we have $\Psi(W) \subseteq V^{\Ort{n - t}}$. 

On $V$, we define the inner product
\[
\langle f_1, f_2 \rangle = \int_{\Ort{n}} \langle f_1(\gamma), f_2(\gamma) \rangle \, d\gamma.
\]
By standard properties of the Haar measure, this makes $V$ a unitary representation of $\Ort{n}$. 
It may be shown that with the inner product chosen in Section~\ref{sec:gl2reps}, the numbers
\[
    \langle \Psi(w_i), \pi(g) \Psi(w_j) \rangle = \int_{\Ort{n}} \langle \Psi(w_i)(\gamma), \Psi(w_j)(\gamma g) \rangle \, d\gamma
\]
are real; see \cite[Section 3]{delaat2023lasserre}. 

For this paper, it is only required that $V$ is a representation of the orthogonal group and $\Psi(W) \subseteq V^{\Ort{n - t}}$. 
However, it follows from the results in \cite{gross1984finite} that the above description is complete in the following sense.
For $n > 2t$, the representations of $\Ort{n}$ defined above are irreducible and we have equality $\Psi(W) = V^{\Ort{n - t}}$. 
Moreover, all irreducible representations of $\Ort{n}$ with nontrivial invariants under $\Ort{n - t}$ are of this form for a unique $\lambda$. 
For $n = 2t$, a complete characterization of the irreducible representations and invariants is given in \cite[Section 8]{gross1984finite}, and using this it may be shown that the description of kernels in our approach is also complete in the case $n = 2t$.
We defer the exact statement and verification to the upcoming PhD thesis of the third-named author.

\subsection{Equivariant functions}\label{sec:defandcont}

In this section, we define a family of $\Ort{n}$-equivariant functions from $\Ii_2$ to the representation $V$ as constructed in Section~\ref{sec:invariantso}. The definition of these functions depends on the choice of representatives of the orbits of $\Ii_2$ under the action of $\Ort{n}$. Let 
\begin{align*}
p_j(\{x,y\}) &= \langle x, y\rangle^j,\\
q_1(\{x,y\}) &= \sqrt{2 + 2\langle x, y\rangle},\\
q_2(\{x,y\}) &= \sqrt{2 - 2\langle x, y\rangle}.
\end{align*}
For the orbit $\Ii_{=0}$ the representative is $\emptyset$ and for the orbit $\Ort{n}J$ with $|J| \geq 1$ we choose the representative
\begin{equation}\label{eq:representatives}
	\left\{\left(\frac{q_1(J)}{2}, \frac{q_2(J)}{2}, 0, \ldots, 0\right), \left(\frac{q_1(J)}{2}, -\frac{q_2(J)}{2}, 0, \ldots, 0\right)\right\}.
\end{equation}
In particular, this means that the standard basis vector $e_1$ is the representative for the orbit $\Ii_{=1}$.

The equivariant functions will be indexed by so-called admissible tuples. If $i = 0$, we call the tuple $(\lambda, i, j, k)$ admissible if $\lambda = (0,0)$, $j = 0$, and $k = 0$.
If $i = 1$, we call the tuple admissible if $\lambda_2 = 0$, $j = 0$, and $k = 0$. Finally, if $i = 2$, we call the tuple admissible for any $\lambda_1 \ge \lambda_2 \ge 0$, $j \geq 0$, and $0 \leq k \leq \lambda_1-\lambda_2$ with $\lambda_2+k$ even.

For each admissible tuple $(\lambda, i, j, k)$, we now define the function
\[
\psi_{\lambda,(i,j,k)}(J) = \xi_{\lambda,i,j,k}(J) \pi(s(J)) \Psi(w_k),
\]
where
\[
\xi_{\lambda,i,j,k}(J) = \begin{cases} 1 & \text{if } i = |J| < 2,\\
p_j(J) q_1(J)^{c_1(k)} q_2(J)^{c_2(k)} & \text{if } i=|J| = 2,\\
0 & \text{otherwise}.\end{cases}
\]
Here $s \colon \Ii_2 \to \Ort{n} $ is a function such that $s(J)R = J$, where $R$ is the orbit representative of the orbit $\Ort{n}J$. To such a function $s$ we shall refer as a section. Once the orbit representatives are fixed, the construction of the functions does not depend on the choice of the section $s$. 

Let us give a brief motivation for these formulae. Firstly, the subscript $i$ indicates the connected component $\Ii_{=i}$ on which the equivariant function is not identically zero. 
The space $\Ii_{=i}$ is homeomorphic to a quotient of $\Ii_{=i}/\Ort{n} \times \Ort{n}/\Ort{n-i}$. For $i=2$, the first factor is homeomorphic to an interval and the second factor to a Stiefel manifold.
The function $p_j$ may be viewed as a function on the factor $\Ii_{=2}/\Ort{n}$ and $\pi(s(J)) \Psi(w_k)$ as a function on the factor $\Ort{n}/\Ort{n - 2}$. These functions are then multiplied to obtain functions on the whole space. The functions $q_1$ and $q_2$ serve two purposes. Namely, they will ensure that we have compatibility with the additional quotient concerning the endpoints of $\Ii_{=2}/\Ort{n}$, and that we obtain polynomial expressions.

\begin{lemma}\label{lem:equivariant}
For admissible $(\lambda,i,j,k)$, the function $\psi_{\lambda,(i,j,k)}$ is equivariant.
\end{lemma}
\begin{proof}
Since $\psi_{\lambda,(i,j,k)}$ is supported on $\Ii_{=i}$, and since the action of $\Ort{n}$ on $\Ii_2$ preserves the cardinality of the sets, we only need to show equivariance for the restriction of $\psi_{\lambda,(i,j,k)}$ to $\Ii_{=i}$. Let $J$ be an element in $\Ii_{=i}$ and let $R$ be the orbit representative of $\Ort{n}J$. 
For $g \in \Ort{n}$, we have $s(gJ) R = gJ$ and $gs(J) R = gJ$, so $s(gJ) = gs(J) h$ for some $h$ in the stabilizer subgroup $\mathrm{Stab}_{\Ort{n}}(R)$. Hence, 
\begin{align*}
\psi_{\lambda,(i,j,k)}(gJ) 
&= \xi_{\lambda,i,j,k}(gJ) \pi(s(gJ)) \Psi(w_k)\\
&= \xi_{\lambda,i,j,k}(J) \pi(gs(J)h) \Psi(w_k)\\ 
&= \xi_{\lambda,i,j,k}(J) \pi(g)\pi(s(J)) \pi(h) \Psi(w_k).
\end{align*}
We will complete the proof by showing that unless $\psi_{\lambda,(i,j,k)}(J)$ and $\psi_{\lambda,(i,j,k)}(gJ)$ are both zero, $\pi(h) \Psi(w_k) = \Psi(w_k)$, which shows 
\[
\psi_{\lambda,(i,j,k)}(gJ) = \pi(g)\psi_{\lambda,(i,j,k)}(J).
\]

For this, we consider the cases $i=0,1,2$ separately. The $i=0$ case is immediate since $\Ii_{=0}$ consists of a single element, and since $\lambda = 0$, $V$ is one dimensional. If $i=1$, then $k=0$, and the stabilizer subgroup of $\Ort{n}$ with respect to $R$ is $\Ort{n-1}$. By formula \eqref{eq:rhomatrixcoeffs}, the dependence of $\rho(\omega \gamma h \epsilon)w_0$ on $\omega \gamma h \epsilon$ is only in the first column, which is equal to the first column of $\omega \gamma \epsilon$, so 
\[
\pi(h) \Psi(w_0)(\gamma) = \rho(\omega \gamma h \epsilon)w_0 = \rho(\omega \gamma \epsilon)w_0 = \Psi(w_0)(\gamma).
\]

If $i=2$ and the points in $J$ are not antipodal, then the stabilizer subgroup of $\Ort{n}$ with respect to $R$ is $S_2 \times \Ort{n-2}$, where $S_2$ is the two-element group generated by the matrix $r$ which maps $e_2$ to $-e_2$ and fixes the orthogonal complement of $e_2$. By construction (see Section~\ref{sec:invariantso}), we have $\pi(h) \Psi(w_k) = \Psi(w_k)$ for $h \in \Ort{n-2}$. The matrix $\omega \gamma r \epsilon$ is the same as $\omega \gamma \epsilon$, except that the second column gets multiplied by $-1$. Since $\lambda_2 + k$ is even, it follows again from formula~\eqref{eq:rhomatrixcoeffs} that $ \rho(\omega \gamma r \epsilon)w_k =  \rho(\omega \gamma \epsilon)w_k$, and thus that 
$\pi(r) \Psi(w_k)(\gamma) = \Psi(w_k)(\gamma)$ holds.

If $i=2$ and the points in $J$ are antipodal, then the stabilizer subgroup is $\Ort{n-1}$ and $q_1(J) = 0$. If $c_1(k) > 0$, then $q_1(J)^{c_1(k)} = 0$, so both $\psi_{\lambda,(i,j,k)}(J)$ and $\psi_{\lambda,(i,j,k)}(gJ)$ are zero. If $c_1(k) = 0$, then $\lambda_2 = 0$ and $k = \lambda_1$, and according to \eqref{eq:rhomatrixcoeffs}, $\rho(\omega \gamma h \epsilon)w_k$ only depends on the second column of $\omega \gamma h \epsilon$, which is equal to the second column of $\omega \gamma \epsilon$, so 
\[
\pi(h) \Psi(w_k)(\gamma) = \rho(\omega \gamma h \epsilon)w_k =  \rho(\omega \gamma \epsilon)w_k = \Psi(w_k)(\gamma).\qedhere
\]
\end{proof}

\subsection{Zonal matrices}

We now define the zonal matrix $Z_\lambda$ by 
\[
Z_\lambda(J_1, J_2)_{(i_1,j_1,k_1), (i_2, j_2, k_2)} =
\langle \psi_{\lambda, (i_1, j_1, k_1)}(J_1), \psi_{\lambda, (i_2, j_2, k_2)}(J_2) \rangle,
\]
where the rows and columns range over all admissible tuples. It follows from equivariance of the function $\psi_{\lambda,(i,j,k)}$ and unitarity of the inner product that the zonal matrices are $\Ort{n}$ invariant.

In the remainder of this section, we use invariant theory to give a short argument showing that the entries of the zonal matrices are polynomials in the inner products between the vectors in $J_1 \cup J_2$. Note that this fact also follows from the direct construction in terms of inner products as given in Section~\ref{sec:speedups}.

\begin{lemma}\label{lem:zonal_to_polynomial3}
    Let $(\lambda, i, j, k)$ be admissible.
	For fixed $w \in W$, the expression
	\[
		\langle w, \psi_{\lambda, (i, j, k)}(\{x_1,\dots,x_i\})(\gamma) \rangle
	\]
	is a polynomial in the entries of the orthogonal matrix $\gamma$ and the vectors $x_1,\dots,x_i$.
\end{lemma}

\begin{proof}
    Given the choice of representatives, we have that for $J=\{x_1\}$, the first column of $s(J)$ is equal to $x_1$, for $J = \{x_1,x_2\}$ with $\langle x_1, x_2\rangle \neq \pm 1$, the first column of $s(J)$ is $(x_1+x_2)/q_1(J)$ and the second column is either $(x_1-x_2)/q_2(J)$ or $(x_2-x_1)/q_2(J)$, and for $J = \{x_1, -x_1\}$ the second column of $s(J)$ is either $x_1$ or $-x_1$. By the choice of admissible tuples, it will turn out that the resulting expressions do not depend on the sign of the second column.
    
	We prove the lemma for each $i$ separately.
    For $i = 0$, the expression is a constant.
    If $i=1$, then $\lambda_2=j=k=0$, and we have
    \[
        \langle w, \psi_{\lambda,(1,0,0)}(\{x_1\}) \rangle = \xi_{\lambda,1,0,0}(\{x_1\}) \langle w, \pi(s(\{x_1\})) \Psi(w_k)(\gamma) \rangle. 
    \]
    Here $\xi_{\lambda,1,0,0}(\{x_1\}) = 1$ and 
    \begin{align*}
        \langle w, \pi(s(\{x_1\})) \Psi(w_0)(\gamma) \rangle &= \langle w, \rho(\omega \gamma s(\{x_1\}) \epsilon) w_0 \rangle.
    \end{align*}
    From the expression \eqref{eq:rhomatrixcoeffs} for the matrix coefficients of $\rho$, it follows that the right-hand side is a polynomial in the entries in the first column of $\omega \gamma s(\{x_1\}) \epsilon$, which is a polynomial in the entries of $\gamma$ and $x_1$.
     
Now let $i = 2$ and set $J = \{x_1,x_2\}$. We will show that
\begin{equation}\label{eqtoshow}
\psi_{\lambda,(i,j,k)}(J) = \langle x_1, x_2\rangle^j \rho(\omega \gamma \mathopen{}\begin{bmatrix}x_1 + x_2 & x_1 - x_2 \end{bmatrix} )  w_k.
\end{equation}
For $\langle x_1, x_2\rangle \neq \pm 1$, we then have
\begin{align*}
\psi_{\lambda,(i,j,k)}(J) &= \xi_{\lambda,i,j,k}(J)\pi(s(J)) \Psi(w_k) \\
&= p_j(J) q_1(J)^{c_1(k)} q_2(J)^{c_2(k)} \rho\mathopen{}\left( \omega \gamma \mathopen{}\begin{bmatrix}\frac{x_1+x_2}{q_1(J)} & \frac{x_1-x_2}{q_2(J)}\end{bmatrix}\right)\mathclose{}w_k.
\end{align*}
Here we used that the expression does not depend on the sign of the second column since $\lambda_2 + k$ is even, i.e., we have 
\[
\rho\mathopen{}\left(\omega \gamma \mathopen{}\begin{bmatrix} u & v \end{bmatrix}\right)\mathclose{}w_k = \rho\mathopen{}\left(\omega \gamma \mathopen{}\begin{bmatrix} u & -v \end{bmatrix}\mathclose{}\right)\mathclose{}w_k
\]
for all orthonormal $u$ and $v$.
Since 
\[
\rho\mathopen{}\left( \omega \gamma \mathopen{}\begin{bmatrix}\frac{x_1+x_2}{q_1(J)} & \frac{x_1-x_2}{q_2(J)}\end{bmatrix}\right)\mathclose{} = 
\rho(\omega \gamma \mathopen{}\begin{bmatrix}x_1 + x_2 & x_1 - x_2 \end{bmatrix} ) \rho\mathopen{}\left(\mathopen{}\begin{bmatrix} 1/q_1(J) & 0 \\ 0 & 1/q_2(J) \end{bmatrix}\right)\mathclose{},
\]
it follows that identity \eqref{eqtoshow} holds whenever $\langle x_1, x_2\rangle \neq \pm 1$.

We will show \eqref{eqtoshow} also holds for the case $x_1 = -x_2$. We have
\begin{align*}
\psi_{\lambda,(i,j,k)}(J) &= \xi_{\lambda,i,j,k}(J)\pi(s(J)) \Psi(w_k) \\
&= p_j(J) q_1(J)^{c_1(k)} q_2(J)^{c_2(k)} \rho( \omega \gamma s(J) \epsilon)w_k.
\end{align*}
We may now substitute $s(J)\epsilon$ with $\begin{bmatrix} c & x_1 \end{bmatrix}$ for any unit vector $c$ orthogonal to $x_1$, to obtain
\begin{align*}
\psi_{\lambda,(i,j,k)}(J)
&= \langle x_1, x_2\rangle^j 0^{c_1(k)} 2^{c_2(k)} \rho\mathopen{}\left( \omega g \mathopen{}\begin{bmatrix} c & x_1 \end{bmatrix}\mathclose{}\right)\mathclose{} w_k\\
&= \langle x_1, x_2 \rangle^j \rho\mathopen{}\left( \omega \gamma \mathopen{}\begin{bmatrix} c & x_1 \end{bmatrix}\mathclose{}\right)\mathclose{}  \rho\mathopen{}\left(\mathopen{}\begin{bmatrix} 0 & 0 \\ 0 & 2 \end{bmatrix}\right)\mathclose{} w_k\\
&= \langle x_1, x_2\rangle^j \rho(\omega \gamma \mathopen{}\begin{bmatrix}x_1 + x_2 & x_1 - x_2 \end{bmatrix} )  w_k.
\end{align*}
A similar argument can be used to show \eqref{eqtoshow} holds for the case $x_1 =x_2$. 
Together this shows 
\[
\langle w, \psi_{\lambda, (i, j, k)}(\{x_1,x_2\})(\gamma) \rangle
\]
is a polynomial in the entries of $\gamma$, $x_1$, and $x_2$.
\end{proof}

\begin{proposition}\label{lem:zonal_to_polynomial2}
	Fix $i_1$ and $i_2$ and let $J_1 = \{x_1,\dots,x_{i_1}\}$ and $J_2 = \{y_1,\dots,y_{i_2}\}$. For admissible tuples $(\lambda,i_1,j_1,k_1)$ and $(\lambda, i_2, j_2, k_2)$, 
	\[
        Z_\lambda(J_1, J_2)_{(i_1,j_1,k_1), (i_2, j_2, k_2)}
	\]
	  is a polynomial in the inner products between the vectors $x_1,\dots,x_{i_1},y_1,\dots,y_{i_2}$.
\end{proposition}
\begin{proof}
By the definition of the inner product on $V$ we have 
\[
Z_\lambda(J_1, J_2)_{(i_1,j_1,k_1), (i_2, j_2, k_2)} = \int_{\Ort{n}}
        \langle \psi_{\lambda, (i_1, j_1, k_1)}(J_1)(\gamma), \psi_{\lambda, (i_2, j_2, k_2)}(J_2)(\gamma) \rangle \, d\gamma.
\]
Since the vectors $w_0,\dots,w_{\lambda_1-\lambda_2}$ form an orthonormal basis of $W$, this is equal to
\[
  \int_{\Ort{n}} \sum_{l=0}^{\lambda_1-\lambda_2} \langle \psi_{\lambda, (i_1, j_1, k_1)}(J_1)(\gamma), w_l \rangle 
  \langle w_l, \psi_{\lambda, (i_2, j_2, k_2)}(J_2)(\gamma)\rangle \, d\gamma.
\]
By Lemma~\ref{lem:zonal_to_polynomial3}, this is a polynomial in the entries of the vectors $x_1,\dots,x_{i_1},y_1,\dots,y_{i_2}$.

By Lemma~\ref{lem:equivariant}, the functions $\psi_{\lambda,(i_1,j_1,k_1)}$ and $\psi_{\lambda,(i_2,j_2,k_2)}$ are equivariant, so by unitarity of the inner product on $V$ it follows that 
\[
Z_\lambda(g J_1, g J_2)_{(i_1,j_1,k_1), (i_2, j_2, k_2)} = Z_\lambda(J_1, J_2)_{(i_1,j_1,k_1), (i_2, j_2, k_2)}
\]
for all $g \in \Ort{n}$. In other words, this is an $\Ort{n}$-invariant polynomial in the vectors $x_1,\dots,x_{i_1},y_1,\dots,y_{i_2}$. By invariant theory (see, e.g., \cite[\S F.1]{fulton91}), it follows that $Z_\lambda(J_1, J_2)_{(i_1,j_1,k_1), (i_2, j_2, k_2)}$ is a polynomial in the inner products between these vectors.
\end{proof}

\section{Efficient computation of the zonal matrices}\label{sec:speedups}

In this section, we explain how we compute the zonal matrices from Section~\ref{sec:zonal}. Throughout we assume $t=2$, but we will sometimes write $t$ instead of $2$ to make explicit the dependence on $t$. Compared to the construction of the zonal matrices in \cite{delaat2023lasserre}, we give a much more efficient approach, which is crucial to be able to perform computations with the truncation degree required to get a sharp bound for the $D_4$ root system.

We have
\begin{align*} 
    & Z_\lambda(J_1, J_2)_{(i_1,j_1,k_1), (i_2, j_2, k_2)}\\
    & \quad = \int_{\Ort{n}} \langle \psi_{\lambda, (i_1, j_1, k_1)}(J_1)(\gamma), \psi_{\lambda, (i_2, j_2, k_2)}(J_2)(\gamma) \rangle \, d\gamma  \\
    & \quad = \xi_{\lambda,i_1, j_1, k_1}(J_1) \xi_{\lambda,i_2, j_2, k_2}(J_2) \int_{\Ort{n}} \langle  \rho(\omega \gamma s(J_1) \epsilon) w_{k_1}, \rho(\omega \gamma s(J_2) \epsilon) w_{k_2} \rangle\, d\gamma\\
    & \quad = \xi_{\lambda,i_1, j_1, k_1}(J_1) \xi_{\lambda,i_2, j_2, k_2}(J_2) P_{k_1, k_2}(s(J_1)^{\sf T} s(J_2)),
\end{align*}
where we define
\begin{equation}\label{eq:PSpoly}
    P_{k_1, k_2}(S) = \int_{\Ort{n}} \langle  \rho(\omega \gamma  \epsilon) w_{k_1}, \rho(\omega \gamma S \epsilon) w_{k_2} \rangle\, d\gamma.
\end{equation}
Here $P_{k_1, k_2}(S)$ is a polynomial in the entries of the $n \times n$ matrix $S$. In this section we show how to compute $P_{k_1,k_2}(S)$ efficiently and how to use this to obtain $Z_\lambda(J_1,J_2)$ as a polynomial in the inner products.

\subsection{Additional symmetries}\label{sec:relations}

To compute $P_{k_1, k_2}(S)$ using the matrix entries of $\rho$, one could directly use the expression
\[
P_{k_1, k_2}(S) = \sum_{l=0}^m \int_{\Ort{n}} \langle  \rho(\omega \gamma  \epsilon) w_{k_1}, w_l \rangle \langle w_l, \rho(\omega \gamma S \epsilon) w_{k_2} \rangle\, d\gamma,
\]
where $m =\lambda_1-\lambda_2$. In this section, we describe additional symmetries under the action of the circle group $\Ort{2}$, which allows us to compute this more efficiently.

Denote by $\rho_\lambda$ the representation of $\GL{2}$ with signature $\lambda$. For any matrix $M$, we have $\rho_\lambda (M) = \det(M)^{\lambda_2} \rho_{(m, 0)}(M)$, and hence
\begin{align*}
P_{k_1, k_2}(S) &= \int_{\Ort{n}} \langle  \rho_\lambda (\omega \gamma  \epsilon) w_{k_1}, \rho_\lambda (\omega \gamma S \epsilon) w_{k_2} \rangle\, d\gamma \\
&= \int_{\Ort{n}} \det (\overline{\omega \gamma  \epsilon})^{\lambda_2}  \langle  \rho_{(m, 0)} (\omega \gamma  \epsilon) w_{k_1}, \rho_{(m, 0)} (\omega \gamma S \epsilon) w_{k_2} \rangle \det (\omega \gamma  S \epsilon)^{\lambda_2} \, d\gamma,
\end{align*}
where $\overline{A}$ denotes the entrywise complex conjugate of $A$.
We introduce some notation to conveniently describe and manipulate expressions such as the one above.
We will refer to the representation $\rho_{(m, 0)}$ as $\rho$ in this section.
Let $\alpha = (\alpha_1, \dots, \alpha_{\lambda_2})$ be a vector with 
\[
\alpha_{i} = (\alpha_{i 1},\alpha_{i 2}) \in \{(1, 1), (1, 2), (2, 1), (2, 2)\},
\]
and let $e$ be the vector with $e_i = (1, 2)$ for all $i$.
We also define the matrices $A = \omega \gamma \epsilon$ and $B = \omega \gamma S \epsilon$.
Denote by $[\alpha]$ the orbit of $\alpha$ under the action of the symmetric group $S_{\lambda_2}$ on the $\lambda_2$ components.
For such $\alpha$ and $0 \leq l_1, l_2 \leq m$ we consider the following polynomial in the entries of $S$:
\begin{equation}\label{eq:elementary_integral}
	J_{l_1, l_2, [\alpha]} = \int_{\Ort{n}} \det (\bar{A})^{\lambda_2}    \rho(A)_{k_1, l_1}^*  \rho(B)_{l_2, k_2}  \prod_{i = 1}^{\lambda_2} B_{\alpha_{i 1}, 1} B_{\alpha_{i 2}, 2} \, d\gamma,
\end{equation}
where $\rho^*$ is the adjoint of $\rho$.

For each signature $\lambda$ that we need and each $0 \leq k_1, k_2 \leq m$, we will show there are coefficients $c_{l_1, k_2, [\sigma]}$, independent of $S$ and $k_1$, such that
\begin{equation}\label{eqcoeffs}
J_{l_1, l_1, [\sigma]} = c_{l_1, k_2, [\sigma]} J_{0, 0, [e]}
\end{equation}
for all $0 \leq l_1 \leq m$ and $\sigma \in \{(1,2), (2,1)\}^{\lambda_2}$.  We will compute these coefficients by solving a linear system for each $\lambda$ and $k_2$.
By expanding both the inner product and the determinant involving $B$, the polynomial $P_{k_1, k_2}(S)$ may then be computed as
\[
    P_{k_1, k_2}(S) = \sum_{l_1, \sigma} (-1)^{s(\sigma)}  J_{l_1, l_1, [\sigma]},
\]
where the sum is over $0 \leq l_1 \leq m$ and all tuples $\sigma \in \{ (1, 2), (2, 1) \}^{\lambda_2}$, and $s(\sigma)$ is the number of times the pair $(2, 1)$ occurs in $\sigma$. By grouping terms and using \eqref{eqcoeffs} we can write this as
\[
P_{k_1, k_2}(S) = J_{0, 0, [e]} \sum_{l_1, [\sigma]} (-1)^{s(\sigma)} \binom{\lambda_2}{s(\sigma)} c_{l_1, k_2, [\sigma]}.
\]
In summary, for fixed $\lambda$, $k_1$ and $k_2$, we need to compute only one integral of the form \eqref{eq:elementary_integral} using this approach.

We now show how to compute these coefficients. 
For $g \in \Ort{t}$, we may substitute $\gamma$ with  $\left( g \oplus g \oplus I_{n - 2t} \right) \gamma$, and this leaves the expression $J_{l_1, l_2, [\alpha]}$ invariant by the invariance property of the Haar measure of $\Ort{n}$.
We have $\omega(g \oplus g \oplus I_{n - 2t}) \gamma = g \omega \gamma$ and hence we may substitute $gA$ for $A$ and $gB$ for $B$.
This gives
\begin{equation}\label{eq:elementary_integral2}
	J_{l_1, l_2, [\alpha]} = \sum_{l_3, l_4, [\beta]}  \det(\bar{g})^{\lambda_2} \rho(\bar{g})_{l_1, l_3} J_{l_3, l_4, [\beta]} \rho(g)_{l_2, l_4} \sum_{\zeta \in [\beta]} \prod_{i = 1}^{\lambda_2} g_{\alpha_{i 1}, \zeta_{i 1}} g_{\alpha_{i 2}, \zeta_{i 2}}.
\end{equation}
We now phrase this in terms of a representation.

Recall that the representation $\Sym^{\lambda_2}(\wedge^2 U) \cong \CC$ is given by multiplication by $\det(g)^{\lambda_2}$. Also recall the representation on $\End(W)$ given by 
\[
g \cdot M = \rho(g) M \rho(g)^*.
\]
For this representation, a basis is given by $w_{l_1} \otimes w_{l_2}^*$. Finally, let $U = \CC^2$ be the representation with the standard action of $\Ort{2}$ and consider the representation $(\phi, \Sym^{\lambda_2}(U^{\otimes 2}))$.
The vectors 
\[
	e_{[\beta]} = \prod_{i = 1}^{\lambda_2} e_{\beta_{i 1}} \otimes e_{\beta_{i 2}}
\]
form a basis.
We consider the inner product such that this basis is orthonormal.
We then consider the dual representation $\phi(g^*)^*$.
We have
\[
	\phi(g^*) e_{[\alpha]} = \prod_{i = 1}^{\lambda_2} g^* e_{\alpha_{i 1}} \otimes g^* e_{\alpha_{i 2}}  = \sum_{[\beta]} \sum_{\zeta \in [\beta]} \prod_{i = 1}^{\lambda_2} g_{\alpha_{i1},\zeta_{i1}} g_{\alpha_{i2},\zeta_{i2}} e_{\beta}
\]
and hence
\[
	\langle	e_\alpha, \phi(g^*)^* e_\beta \rangle = \langle	\phi(g^*) e_\alpha, e_\beta \rangle = \sum_{\zeta \in [\beta]} \prod_{i = 1}^{\lambda_2} g_{\alpha_{i1},\zeta_{i1}} g_{\alpha_{i2},\zeta_{i2}}.
\]

Tensoring the above representations gives the representation 
\[
(\Phi, \Sym^{\lambda_2}(\wedge {}^2 U) \otimes \End(W) \otimes \Sym^{\lambda_2}(U^{\otimes 2}))
\]
and a basis is given by $e_{l_1, l_2, [\alpha]} = w_{l_1} \otimes w_{l_2}^* \otimes e_{[\alpha]} $.
We get
\begin{equation}\label{eq:bigin}
\langle e_{l_1, l_2, [\alpha]}, \Phi(g) e_{l_3, l_4, [\beta]}\rangle =  \det(g)^{\lambda_2}\rho(g)_{l_1, l_3} \rho(g)_{l_2, l_4} \sum_{\zeta \in [\beta]} \prod_{i = 1}^{\lambda_2} g_{\alpha_{i 1}, \zeta_{i 1}} g_{\alpha_{i 2}, \zeta_{i 2}}.
\end{equation}
Using \eqref{eq:elementary_integral2} and \eqref{eq:bigin} we have
\begin{align*}
	&\Phi(g) \sum_{l_3, l_4, [\beta]} J_{l_3, l_4, [\beta]} e_{l_3, l_3, [\beta]}\\
 &\quad= \sum_{l_1, l_2, [\alpha]} \sum_{l_3, l_4, [\beta]} J_{l_3, l_4, [\beta]} \langle e_{l_1, l_2, [\alpha]}, \Phi(g)  e_{l_3, l_4, [\beta]}  \rangle e_{l_1, l_2, [\alpha]}\\
	&\quad= \sum_{l_1, l_2, [\alpha]} \sum_{l_3, l_4, [\beta]}  \det(\bar{g})^{\lambda_2} \rho(\bar{g})_{l_1, l_3} J_{l_3, l_4, [\beta]} \rho(g)_{l_2, l_4} \sum_{\zeta \in [\beta]} \prod_{i = 1}^{\lambda_2} g_{\alpha_{i 1}, \zeta_{i 1}} g_{\alpha_{i 2}, \zeta_{i 2}} e_{l_1, l_2, [\alpha]}\\
 &\quad= \sum_{l_1, l_2, [\alpha]} J_{l_1, l_2, [\alpha]} e_{l_1, l_2, [\alpha]}.
\end{align*}
Defining
\[
J = \sum_{l_1, l_2, [\alpha]} J_{l_1, l_2, \alpha} e_{l_1, l_2, [\alpha]},
\]
this equation is expressed as $\Phi(g)J=J$ for all $g \in \Ort{2}$.
Using the exponential map, this is equivalent to the condition $d \Phi(X) J = 0$, where
\[
X = \begin{bmatrix}
0 & 1 \\
-1 & 0
\end{bmatrix},
\]
 and $\Phi(g_0) J = J$,  where $g_0$ is an orthogonal matrix with $\det(g_0) = -1$.
This follows from the fact that $X$ spans the Lie algebra $\mathfrak{so}(2)$.
The additional condition with $g_0$ comes from the fact that the equation has to hold for all orthogonal matrices and not merely for the special orthogonal matrices.

We now write out the system $d \Phi(X) J = 0$ in components.
Let $g(t)$ be a curve of special orthogonal matrices such that $g(0) = I$ and $g'(0) = X$. 
To obtain the components of $d \Phi (X)$, we plug  $g(t)$ into \eqref{eq:bigin} and take the derivative. 
Using the product rule, one obtains
\begin{align*}
	&\langle e_{l_1, l_2, [\alpha]}, d \Phi(X) e_{l_3, l_4, [\beta]}\rangle \\
 &\quad =   d\rho(X)_{l_1, l_3} \delta_{l_2, l_4} \delta_{[\alpha], [\beta]} +  \delta_{l_1, l_3} d\rho(X)_{l_2, l_4} \delta_{[\alpha], [\beta]} +  \delta_{l_1, l_3} \delta_{l_2, l_4} G'(0),
\end{align*}
where we have defined
\[
	G(t) = \sum_{\zeta \in [\beta]} \prod_{i = 1}^{\lambda_2} g(t)_{\alpha_{i 1}, \zeta_{i 1}} g(t)_{\alpha_{i 2}, \zeta_{i 2}}.
\]
A formula for $d\rho(X)$ can be found in Section~\ref{sec:gl2reps}. One may further verify that each term of $G'(0)$ is zero unless $\alpha_{ij}$ and $\zeta_{ij}$ differ for exactly one $ij$, in which case the term equals $X_{\alpha_{ij}, \zeta_{ij}}$. Together this gives explicit formulas for the linear constraints on the coefficients $J_{l_1,l_2,[\alpha]}$ arising from $d\Phi(X) J = 0$.

We now work out the condition $\Phi(g_0)J = J$. For this, we let $d_j(\alpha)$ be the total number of occurrences of $j$ in $\alpha$.
Recall the signature of $\rho$ is $(m, 0)$, so that we have $c_1(l)  = m - l$ and $c_2(l) = l$.

\begin{lemma}\label{lemma:zero_integral2}
	If $\lambda_2 + c_2(l_1) + c_2(l_2) + d_2(\alpha)$ is odd, then $ J_{l_1, l_2, [\alpha]} = 0 $. 
\end{lemma}
\begin{proof}
Let $g_0 = \begin{bmatrix}
1 & 0 \\
0 & -1
\end{bmatrix} $.
We then have 
\[
	\langle e_{l_1, l_2, [\alpha]}, \Phi(g_0) e_{l_3, l_4, [\beta]}\rangle = \delta_{l_1, l_3} \delta_{l_2, l_4} \delta_{[\alpha] ,[\beta]}  (-1)^{\lambda_2 + c_2(l_1) + c_2(l_2) + d_2(\alpha)} 
\]
and hence from $J = \Phi(g_0)J$ we obtain
\[
    J_{l_1, l_2, [\alpha]} = (-1)^{\lambda_2 + c_2(l_1) + c_2(l_2) + d_2(\alpha)} J_{l_1, l_2, [\alpha]}.\qedhere
\]
\end{proof}

We give additional conditions under which $J_{l_1,l_2,[\alpha]}$ vanishes.

\begin{lemma}\label{lemma:zero_integral}
	Let $j \in \{1,2\}$. If $c_j(l_1) + \lambda_2 - (c_j(l_2) + d_j(\alpha)) \neq 0$, then $ J_{l_1, l_2, [\alpha]} = 0 $. 
\end{lemma}
\begin{proof}
Let $R(\theta)$ be the matrix rotating the $j$ and $j+t$ rows of $\gamma$ by
\[
\begin{bmatrix}
\cos(\theta) & -\sin(\theta) \\
\sin(\theta) & \cos(\theta)
\end{bmatrix}.
\]
We then have $ \omega R(\theta)  = A(\theta) \omega $, where $A(\theta)$ is the diagonal matrix with $e^{i\theta}$ at the $j$th diagonal entry  and $1$ at the other diagonal entry. The matrix $R(\theta)$ is orthogonal and by a similar argument as before we may substitute $\omega$ with $A(\theta) \omega$.  
We then obtain that $J_{l_1, l_2, [\alpha]}$ is equal to
\[
 \sum_{l_3, l_4, [\beta]} \det(\overline{A(\theta)})^{\lambda_2} \rho(\overline{A(\theta)})_{l_1,l_3} J_{l_3, l_4, [\beta]} \rho(A(\theta))_{l_2, l_4} \prod_{i = 1}^{\lambda_2} A(\theta)_{\alpha_{i 1} \beta_{i 1}} A(\theta)_{\alpha_{i 2} \beta_{i 2}}.
 \]
Working this out gives
\[
	J_{l_1, l_2, [\alpha]}= e^{- i \theta(c_j(l_1) + \lambda_2 - (c_j(l_2) + d_j(\alpha)) )} J_{l_1, l_2, [\alpha]}.
 \]
Since this equation holds for all $\theta$, we have $J_{l_1, l_2, [\alpha]} = 0$.
\end{proof}

We thus have the system $d \Phi(X) J = 0$ and certain components of $J$ vanish due to Lemmas \ref{lemma:zero_integral2} and \ref{lemma:zero_integral}. 
As a final step, which is necessary to ensure the solution space is one-dimensional, we add the following relations. 
By expanding into the monomials $B_{11}$, $B_{12}$, $B_{21}$, and $B_{22}$, there are coefficients $a_{l_2,k_2,[\alpha],\mu}$ such that
\[
	\rho(B)_{l_2, k_2}  \prod_{i = 1}^{\lambda_2} B_{\alpha_{i 1}, 1} B_{\alpha_{i 2}, 2} = \sum_{\mu} a_{l_2, k_2, [\alpha], \mu} B^{\mu}.
\]
With 
\[
	K_{l_1, \mu} = \int_{\Ort{n}} \det (\bar{A})^{\lambda_2} \rho(A)_{k_1, l_1}^* B^{\mu} \,d\gamma
\]
we have
\[
	J_{l_1, l_2, [\alpha]} 
 = \sum_{\mu} a_{l_2, k_2, [\alpha], \mu} K_{l_1, \mu}.
\]
We now enlarge the linear system by introducing new variables for the $K_{l_1,\mu}$, and for each $l_1$, $l_2$, and $\alpha$ we add the above constraint on the variables $J_{l_1,l_2,[\alpha]}$ and $K_{l_1,\mu}$.

Finally, we project the linear space satisfying all of the above relations to the space 
\[
\mathrm{span}\{e_{l_1, l_1, [\sigma]} \mid 0 \leq l_1 \leq m, \, \sigma \in \{(1,2),(2,1)\}^{\lambda_2}\}.
\]
For this, we consider the homogeneous linear system given by the constraints discussed above. We order the columns so that the variables corresponding to $J_{l_1,l_1,[\sigma]}$ are at the end, and $J_{0,0,[e]}$ corresponds to the final column. Then we perform row reduction using rational arithmetic and find that the final column is the only free variable among the columns corresponding to the variables $J_{l_1,l_1,[\sigma]}$. From this, we find the coefficients $c_{l_1, k_2, [\sigma]}$ for which \eqref{eqcoeffs} holds.

\subsection{Real parts}
\label{sec:realparts}

As shown in Section~\ref{sec:relations}, to compute the zonal matrices we need to compute the quantity  
\[
    J_{0, 0, [e]} = \int_{\Ort{n}} \det (\overline{\omega} \gamma \epsilon)^{\lambda_2} \rho(\omega \gamma \epsilon)^*_{k_1, 0} \rho(\omega \gamma S \epsilon)_{0, k_2} ((\omega \gamma S \epsilon)_{1,1}(\omega \gamma S \epsilon)_{2,2})^{\lambda_2}\,d\gamma.
\]
In this section we will show that for $\lambda_1 > 0$, this is equal to
\begin{equation}\label{eq:realparts}
    2 \int_{\Ort{n}} \mathcal{R}\mathopen{}\left( \det (\overline{\omega} \gamma \epsilon)^{\lambda_2} \rho(\omega \gamma \epsilon)^*_{k_1, 0} \right)\mathclose{} \mathcal{R}\mathopen{}\left(  \rho(\omega \gamma S \epsilon)_{0, k_2} ((\omega \gamma S \epsilon)_{1,1} (\omega \gamma S \epsilon)_{2,2})^{\lambda_2} \right)\mathclose{}\,d\gamma,
\end{equation}
where $\mathcal{R}(z)$ denotes the real part of $z \in \mathbb{C}$. This yields a factor two speedup in the most expensive part of the generation of the zonal matrices.

For a matrix $M$ and a vector $a$ of natural numbers of the same size, let us adopt the notation
\[
    M^a = \prod_{i,j} M_{i,j}^{a_{i,j}}.
\]
By multilinearity and the formula for the matrix coefficients of the representations of $\GL{2}$, it suffices to show
\begin{equation} \label{eq:integral1}  
    \int_{\Ort{n}} ( \overline{\omega} \gamma \epsilon )^a (\omega \gamma S \epsilon)^b\, d\gamma = 2 \int_{\Ort{n}} \Rr  \left( ( \overline{\omega} \gamma \epsilon )^a \right) \Rr \left( ( \omega \gamma S \epsilon )^b \right)\, d\gamma
\end{equation}
for all $a, b \in \N^{2 \times 2}$ with $|a| = |b| = |\lambda|$.

To show this, we introduce the variables $R_{11}$, $R_{12}$, $R_{21}$, and $R_{22}$, the matrices
\[ R^+_k = 
    \mathopen{}\begin{bmatrix}
        R_{k1} I_2  & R_{k2} I_2  & 0
    \end{bmatrix}\mathclose{},  
\] 
and the vectors $R_k = \begin{bmatrix} R_{k1} & R_{k2}\end{bmatrix}$ for $k = 1, 2$. 
We then consider the polynomial
\begin{align}
    \int_{\Ort{n}} ( R^+_{1} \gamma \epsilon )^a (R^+_{2} \gamma S \epsilon)^b\, d\gamma 
    &= \sum_{\substack{ |u| = |v| = |\lambda|}} I_{u, v} R_{1}^{u} R_{2}^{v} \label{eq:poly1} \\  
    &= \sum_{|s| = 2|\lambda|} \sum_{\substack{u+v = s\\ |u| = |v| = |\lambda|}} I_{u, v} R_{11}^{u_1} R_{12}^{u_2} R_{21}^{v_1}R_{22}^{v_2}, \nonumber
\end{align}
where the real numbers $I_{u, v}$ are obtained by working out brackets and gathering terms.

Substituting $R_{11} = 1$, $R_{12} = -i$, $R_{21} = 1$ and $R_{22} = i$ gives the left-hand side of \eqref{eq:integral1}, which is a real number by \cite[Section 3]{delaat2023lasserre}. So the sum over all terms with $u_2+v_2$ odd vanishes.
Since $|s|$ is even, $u_1 + v_1$ is restricted to be even too.
We reparametrize the sum and obtain that the left-hand side of (\ref{eq:integral1}) is given by
\begin{equation} \label{eq:integral2}
    \sum_{|s| = |\lambda|} (-1)^{s_2} \sum_{\substack{u + v = 2s \\ |u| = |v| = |\lambda|}} (-1)^{u_2} I_{u, v}.
\end{equation}
A similar reasoning shows that the right-hand side of \eqref{eq:integral1}  is equal to 
\[
    2\sum_{|s| = |\lambda|} (-1)^{s_2} \sum_{\substack{u + v = 2s\\u_2\, \mathrm{even} \\ |u| = |v| = |\lambda|}} I_{u, v}.
\]

We now substitute $R_1 = R_2$ in (\ref{eq:poly1}) to obtain the polynomial
\begin{equation}\label{eq:poly4} 
\begin{split}
     \int_{\Ort{n}} ( R^+_{1} \gamma \epsilon )^a (R^+_{1} \gamma S \epsilon)^b \,d\gamma     &= \sum_{\substack{ |u| = |v| = |\lambda|}} I_{u, v} R_{1}^{u} R_{1}^{v} \\
    &= \sum_{|s| = 2|\lambda|} \sum_{\substack{u+v = s\\ |u| = |v| = |\lambda|}} I_{u, v} R_1^s.
    \end{split}
\end{equation}
Similarly as before, we may substitute $\gamma$ with $\left( g \oplus g \oplus I_{n - 2t} \right) \gamma$, and this leaves the polynomial \eqref{eq:poly4} invariant by the invariance property of the Haar measure of $\Ort{n}$. 
We have $R_1^+ \left( g \oplus g \oplus I_{n - 2t} \right) \gamma = g R_1^+ \gamma$. 
Hence polynomial \eqref{eq:poly4} is a polynomial in $R_{11}^2 + R_{12}^2$ by invariant theory.
Since it is also a homogeneous polynomial of total degree $2|\lambda|$, it must be linearly proportional to the polynomial
\begin{equation}\label{eq:invpoly1}
    \left( R_{11}^2 + R_{12}^2 \right)^{|\lambda|}.
\end{equation}
Hence the $s$ which occur in the sum in \eqref{eq:poly4} must have even entries, and the polynomial \eqref{eq:poly4} may be written as
\[
    \sum_{|s| = |\lambda|} \sum_{\substack{u+v = 2s\\ |u| = |v| = |\lambda|}} I_{u, v} R_1^{2s} = \sum_{|s| = |\lambda|} c_s R_1^{2s}.
\]
Furthermore, since it must be linearly proportional to \eqref{eq:invpoly1}, we have 
\[
c_s = \binom{|\lambda|}{s_2} c_0
\]
by the binomial theorem.
We now rearrange terms to obtain
\begin{align*}
    \sum_{u+v = 2s} (-1)^{u_2} I_{u, v} 
    &=  \sum_{\substack{u+v = 2s \\ u_2 \textup{ even}}} (-1)^{u_2} I_{u, v} + \sum_{\substack{u+v = 2s \\ u_2 \textup{ odd}}} (-1)^{u_2} I_{u, v}  \\ 
    &= 2 \sum_{\substack{u+v = 2s \\ u_2 \textup{ even}}} (-1)^{u_2} I_{u, v} - \sum_{\substack{u+v = 2s}} I_{u, v}  \\ 
    &= 2 \sum_{\substack{u+v = 2s \\ u_2 \textup{ even}}} (-1)^{u_2} I_{u, v} - c_s, 
\end{align*}
where in each sum we implicitly assume $|u| = |v| = |\lambda|$.
Using (\ref{eq:integral2}), we now see that we may write the left-hand side of (\ref{eq:integral1}) as 
\begin{align*}
    \sum_{|s| = |\lambda|} (-1)^{s_2} \sum_{u+v = 2s} (-1)^{u_2} I_{u, v}
    &= 2\sum_{|s| = |\lambda|} (-1)^{s_2} \sum_{\substack{u+v = 2s \\ u_2 \text{ even}}} I_{u, v} - \sum_{|s| = |\lambda|} (-1)^{s_2} c_s  \\
    &= 2\sum_{|s| = |\lambda|} (-1)^{s_2} \sum_{\substack{u+v = 2s \\ u_2 \text{ even}}} I_{u, v},
\end{align*}
since
\[
\sum_{|s| = |\lambda|} (-1)^{s_2}c_s = c_0\sum_{s_2 =0}^{|\lambda|}\binom{|\lambda| }{s_2} (-1)^{s_2} =  c_0(1-1)^{|\lambda|} = 0
\]
whenever $|\lambda| > 0$.
Recall that the sum over even $u_2$ equals the integral of the product of the real parts. Hence we have shown equation \eqref{eq:integral1}, which is what we wanted to show.

\subsection{Inner products}\label{sec:computational_aspects}

In this section, we describe how to compute
\[
Z_{\lambda}(J_1, J_2)_{(i_1, j_1, k_1), (i_2, j_2, k_2)}
\]
efficiently as a polynomial in the inner products between the vectors in $J_1 \cup J_2$.

We have 
\[
\rho(\omega \gamma h \epsilon) w_{k_1} = \rho(\omega \gamma \epsilon) w_{k_1}
\]
for all $h \in \Ort{n-2}$, where as before we view $h$ as a matrix in $\Ort{n}$ fixing the first $2$ coordinates.
Let $P_{k_1,k_2}$ be the polynomial defined in \eqref{eq:PSpoly}. By the invariance property of the Haar measure, we have 
\[
P_{k_1,k_2}(hS) = P_{k_1, k_2}(S)
\]
for all $h \in \Ort{n-2}$. 
Let $S_1$ be the top-left $2 \times 2$ block of $S$ and $S_2$ the bottom-left $(n-2) \times 2$ block of $S$.
Using invariant theory (see, e.g., \cite[\S F.1]{fulton91}), we see that $P_{k_1, k_2}(S)$ must be a polynomial in the entries of $S_1$ and the inner products between the columns of $S_2$. 

Consider the ideal $\mathcal J$ in $\RR[S]$ generated by the entries of  $S^{\sf T}S - I$ and the monomials $S_{i,j}$ for $i > 2 +j$. 
Consider the polynomials in $\mathcal J$ given by
\begin{gather*}
     S_{i, j} \textup{ with } i > 2 + j \text{ and } j \leq 2,  \quad
     1 - \sum_{i = 1}^{4} S^2_{i, 2}, \quad \sum_{i = 1}^{3} S_{i, 1} S_{i, 2}, \quad 1 - \sum_{i = 1}^{3} S_{i, 1}^2. 
\end{gather*}
We now use the above polynomials to perform Euclidean division on $P_{k_1, k_2}$ and show the remainder $p_{k_1,k_2}$ is a polynomial in the entries of $S_1$. Here we use a lexicographical term order where the variables are ordered such that $S_{4,2} > S_{3,2} > S_{3,1}$ and $S_{i,j} < S_{3,1}$ if $i\le 2$ and $S_{i,j} > S_{4,2}$ if $i>2+j$. This results in the following concrete procedure.
We first remove the terms in $P_{k_1, k_2}$ that contain a variable $S_{i, j}$ with $i > 2 + j$, after which we obtain a polynomial of the form
\begin{equation*}\label{eq:firststage}
    \sum_{\alpha, a, b, c} C_{\alpha, a, b, c} S_1^\alpha S_{3, 1}^{2a} ( S_{3, 1} S_{3, 2})^b (S_{3, 2}^2+S_{4,2}^2)^c
\end{equation*}
for some $C$, $\alpha$, $a$, $b$, and $c$.
In this polynomial, we first replace every occurrence of $S_{4,2}^2$ with $1 - S_{1, 2}^2 - S_{2, 2}^2 - S_{3, 2}^2$, then every occurrence of $S_{3,1} S_{3, 2}$ with $- S_{1,1} S_{1, 2}-S_{2,1} S_{2, 2}$, and finally every occurrence of $S_{3, 1}^2$ with $1 - S_{1, 1}^2 - S_{2, 1}^2$.
This gives a polynomial $p_{k_1,k_2}$ in the entries of $S_1$. At each step we have subtracted elements of $\mathcal{J}$, so
\[
P_{k_1, k_2} - p_{k_1,k_2} \in \mathcal J.
\]

From formula \eqref{eq:rhomatrixcoeffs} for the matrix coefficients of the representation $\rho$ of $\GL{2}$, it follows that every monomial in the expansion of $P_{k_1,k_2}(S)$ contains $c_1(k_2)$ variables 
from the first column of $S$ and $c_2(k_2)$ variables from the second column of $S$. Furthermore, in each step of the procedure to obtain $p_{k_1,k_2}$ from $P_{k_1, k_2}$, the number of variables in each monomial from a given column stays the same or drops by an even number. This shows that in each monomial in $p_{k_1,k_2}(S)$, the number of variables from column $l \in \{1,2\}$ is at most $c_l(k_2)$, and differs from this by an even number. 

We would like to say something similar about the number of variables from each row.
For all $g \in \Ort{n}$ we have
\begin{align}\label{eq:Ptransposetranspose}
\begin{split}
    P_{k_1, k_2} (g) &= \langle \Psi(w_{k_1}), \pi(g) \Psi(w_{k_2}) \rangle \\    
    &= \langle \pi(g^{\sf T}) \Psi(w_{k_1}), \Psi(w_{k_2}) \rangle \\   
   % &= \overline{\langle  \Psi(w_{k_2}),\pi(g^{\sf T}) \Psi(w_{k_1}) \rangle }\\   
    &= P_{k_2, k_1}(g^{\sf T}),
\end{split}
\end{align}
where we used  the fact that the inner product is unitary and $P_{k_2, k_1}(g^{\sf T})$ is real.
For each $g \in \Ort{n}$, there is an element $h \in \Ort{n-2}$ such that $(hg)_{i,j} = 0$ for $i > 2+j$. Hence,
\[
P_{k_1, k_2}(g) = P_{k_1, k_2}(hg) = p_{k_1,k_2}(h g) = p_{k_1,k_2}(g),
\]
where the second equality holds because $hg$ lies in the vanishing locus of $\mathcal J$. Using \eqref{eq:Ptransposetranspose}, it follows that
$
p_{k_1,k_2}(g) = p_{k_2,k_1}(g^{\sf T})
$
for all $g \in \Ort{n}$. 
The only variables which occur in $p_{k_1,k_2}(S)$ and $p_{k_2,k_1}(S^{\sf T})$ are from the top-left $2 \times 2$ block of $S$ and so we may view them as functions on $\RR^4$. 
Consider the subset of $\RR^4$ given by projecting $\Ort{n}$ to the top-left $2 \times 2$ block. 
Since
$
p_{k_1,k_2}(g) = p_{k_2,k_1}(g^{\sf T})
$
for all $g \in \Ort{n}$, the functions agree on this subset.
This subset has a nonempty interior, and hence the polynomials agree on $\RR^4$.
Hence we have equality of polynomials: $p_{k_1,k_2}(S) = p_{k_2,k_1}(S^{\sf T})$.
This shows that in each monomial in $p_{k_1,k_2}(S)$, the number of variables from row $l \in \{1,2\}$ is at most $c_l(k_1)$, and differs from this by an even number.

Since $c_l(k_1) = 0$ for $l > i_1$ and $c_l(k_2) = 0$ for $l > i_2$, it follows that $p_{k_1,k_2}(S)$ is a polynomial in the top-left $i_1 \times i_2$ block.
As discussed in the proof of Lemma~\ref{lem:zonal_to_polynomial3}, the $(l_1, l_2)$ entry, with $1 \le l_1 \le i_1$ and $1 \le l_2 \le i_2$, of $s(J_1)^{\sf T}s(J_2)$ has denominator $q_{l_1}(J_{1}) q_{l_2}(J_{2})$. 
To obtain the zonal matrix entry, we may replace each monomial $S^a$ in $p_{k_1,k_2}(S)$ with
\begin{align*}
    & \xi_{\lambda, i_1, j_1, k_1}(J_1) \xi_{\lambda, i_2, j_2, k_2}(J_2)(s(J_1)^{\sf T}s(J_2))^a \\
    & \quad = q_1(J_1)^{c_1(k_1)} q_2(J_1)^{c_2(k_1)} q_1(J_2)^{c_1(k_2)} q_2(J_2)^{c_2(k_2)} (s(J_1)^{\sf T}s(J_2))^a,
\end{align*}
and by the properties of $p_{k_1,k_2}$ as discussed above, this is a polynomial in the entries of the vectors in $J_1 \cup J_2$. From this we can easily read of the polynomial in terms of the inner products between these vectors.

We now describe additional techniques to speed up the implementation.
By Section~\ref{sec:relations} and \ref{sec:realparts}, the integrand of $P_{k_1, k_2}(S)$ may be replaced by the product of 
\begin{equation}\label{eq:realpart_left}
\mathcal{R}(\det(\overline \omega \gamma\epsilon)\rho(\overline \omega \gamma \epsilon)_{k_1,0}^*)
\end{equation}
and
\begin{equation}\label{eq:realpart_right}
\mathcal{R}(\rho(\omega \gamma S \epsilon)_{0, k_2} ((\omega \gamma S \epsilon)_{1,1} (\omega \gamma S \epsilon)_{2,2})^{\lambda_2}).
\end{equation}
The integration over $\Ort{n}$ and the substitution procedure described above may be swapped.
We first compute \eqref{eq:realpart_right} explicitly as a polynomial in the variables $S_{i,j}$ with $i \leq 2 + j$ and $j \leq 2$ and the top-left $4 \times 4$ block of $\gamma$. We then perform the above substitution procedure. Since we know that after integration over $\Ort{n}$ all terms with variables from $S_2$ will vanish, we remove those terms. This gives a polynomial in the top-left $i_1 \times i_2$ block of $S$ and the top-left $4 \times 4$ block of $\gamma$.

Whenever $\sum_i a_{ij}$ or $\sum_{i} a_{ji}$ is odd for any $j$, we have 
\[
\int_{\Ort{n}}\gamma^a \, d\gamma = 0.
\]
This means that we do not have to work out the product of the whole polynomial \eqref{eq:realpart_right} with \eqref{eq:realpart_left}. 
Instead, we only multiply terms that produce monomials in $\gamma$ which do not immediately vanish. We then integrate each monomial in $\gamma$ using the recursion formulas of \cite{gorinlopez2008}. This enables us to explicitly compute $p(S)$, from which we obtain the zonal matrix entry as explained above.

\section{Semidefinite programming formulation}
\label{sec:sdp}

Let $d_1 \leq d_2 \leq \delta$ be positive integers with $\delta$ even. In our application to the $D_4$ root system, we use $d_1=14$ and $d_2 = \delta=16$.

In the semidefinite program, we optimize over positive semidefinite matrices $\widehat K_\lambda$. Here the rows and columns are indexed by tuples $(i,j,k)$ for which $(\lambda,i,j,k)$ is admissible (see Section~\ref{sec:defandcont}) and $|\lambda| + 2j \leq d_2$, and we similarly restrict the rows and columns of $Z_\lambda$.
Let
\[
K(J_1, J_2) = \sum_{|\lambda| \le d_1} \langle \widehat K_\lambda,  Z_\lambda(J_1,J_2) \rangle.
\]

It follows from Proposition~\ref{lem:zonal_to_polynomial2} that $A_2K(Q)$ is a polynomial in the inner products between the vectors in $Q$. Using Section~\ref{sec:speedups}, we can find polynomials $p_1,\dots,p_4$ in $0$, $1$, $3$, and $6$ variables, such that
\begin{align*}
p_1 &= A_2K(\{x_1\}),\\
p_2(\langle x_1, x_2\rangle) &= A_2K(\{x_1, x_2\}),\\
p_3(\langle x_1, x_2\rangle, \langle x_1, x_3\rangle, \langle x_2, x_3\rangle) &= A_2K(\{x_1, x_2, x_3\}),\\
p_4(\langle x_1, x_2\rangle, \langle x_1, x_3\rangle, \ldots, \langle x_3, x_4\rangle) &= A_2K(\{x_1, \dots, x_4\}).
\end{align*}
Here $p_3$ is $S_3$-invariant and $p_4$ is $S_4$-invariant, where $S_3$ acts by permuting variables and the action of $S_4$ is such that 
\begin{align*}
&p_4(\langle x_{\sigma(1)}, x_{\sigma(2)}\rangle, \langle x_{\sigma(1)}, x_{\sigma(3)}\rangle, \ldots, \langle x_{\sigma(3)}, x_{\sigma(4)}\rangle)\\
&\quad= p_4(\langle x_1, x_2\rangle, \langle x_1, x_3\rangle, \ldots, \langle x_3, x_4\rangle).
\end{align*}
for all $\sigma \in S_4$. Note that the polynomials $p_1,\dots,p_4$ are of degree at most $d_2$, and their coefficients depend linearly on the entries of the matrices $\widehat K_\lambda$.

The Fourier truncated version of \eqref{pr:las} can be formulated as the following semidefinite program with polynomial inequality constraints:
\begin{equation}\label{pr:semialgebraic}
\begin{aligned}
& \text{minimize} & & (\widehat K_0)_{(0,0,0),(0,0,0)}\\
& \text{subject to} & & \widehat K_\lambda \succeq 0, && |\lambda| \leq d_1,\\
& & & p_1 \leq -1,\\
& & & p_2(u) \leq 0,&& u \in [-1,\cos\theta],\\
& & & p_3(u_1,u_2,u_3) \leq 0, && (u_1,u_2,u_3) \in \Delta_3,\\
& & & p_4(u_1,\dots,u_6) \leq 0, && (u_1,\dots,u_6) \in \Delta_4.
\end{aligned}
\end{equation}
Let
\[
G_3 = \begin{pmatrix} 1 & u_1 & u_2\\ u_1 & 1 & u_3 \\ u_2 & u_3 & 1\end{pmatrix} \quad\text{and}\quad G_4=  \begin{pmatrix} 1 & u_1 & u_2 & u_3 \\ u_1 & 1 & u_4 & u_5 \\ u_2 & u_4 & 1 & u_6\\ u_3 &u_5 & u_6 & 1\end{pmatrix}.
\]
The semialgebraic set $\Delta_i$ consists of all $u \in \R^{i \choose 2}$ with
$(u_j+1)(\cos\theta-u_j)\geq 0$ for all $1 \leq j \leq \binom{i}{2}$ and for which the determinants of all principal submatrices of size at least $3$ of $G_i$ are nonnegative.

We can now use sum-of-squares polynomials to relax this further to a semidefinite program. 
For the two-point constraint, for instance, we can use  Luk\'acs result (see, e.g., \cite{powers2000polynomials}) to replace the condition $p_2(u) \le 0$ for $u \in [-1,\cos\theta]$ by 
\begin{equation}\label{eq:lukacs}
p_2(u) + s_0(u) + (u+1)(\cos\theta-u) s_1(u) \equiv 0,
\end{equation}
where $s_0$ and $s_1$ are sum-of-squares polynomials of degree $\delta$ and $\delta-2$, respectively. Let $m_l(u)$ be a vector whose entries form a basis for the polynomials up to degree $l$. We can write 
\[
s_k(u) = \langle m_{\delta/2-k}(u)  m_{\delta/2-k}(u)^{\sf T}, M_k \rangle,
\]
where $M_k$ is a positive semidefinite matrix. In this way, we can replace the two-point polynomial inequality constraint by two positive semidefinite matrices and several linear constraints that enforce the polynomial identity \eqref{eq:lukacs}.

We can do something similar for the three-point and four-point constraints. Suppose $\Delta_i = \{u : g_k(u) \ge 0, \, k=1,\ldots,l\}$, then we relax the polynomial inequality constraint to the identity
\[
p_i(u) + \sum_{k=0}^l r_k(u) g_k(u) \equiv 0,
\]
where we set $g_0(u) = 1$ and $r_k(u)$ is a sum-of-squares polynomial of degree at most $\delta - \deg(g_k)$. By Putinar's theorem \cite{putinar93} (see also \cite[Chapter 13]{marshall2008positive}), this relaxation converges to the original polynomial constraint when $\delta \to \infty$.

In the resulting semidefinite program, the positive semidefinite matrix variables for the four-point constraint will be far larger than any other matrix in the program. For this reason, exploiting the symmetries in the polynomials is essential.

If a semialgebraic set is invariant under the action of a group, then there exists a description in terms of invariant polynomials \cite{MR1677398}. Let 
\[
\{q_1, \dots, q_l\}
\]
be an orbit of the polynomials describing $\Delta_i$ under the action of $S_i$. Then we can replace the polynomials in this orbit by the $S_i$-invariant polynomials 
\[
\sum_{\substack{B\subseteq \{1, \ldots, l\} \\ |B| = b}}\prod_{k \in B} q_k
\]
for $b=1, \ldots, l$; see, e.g., \cite{machado16,leijenhorst2023solving} for the proof.

We may now assume the sum-of-squares polynomials for the three-point and four-point constraints are also invariant under the given action of the symmetric groups $S_3$ and $S_4$. This means that instead of using one large positive semidefinite matrix, we can use several smaller positive semidefinite matrices to model each sum-of-squares polynomial \cite{gatermann04}. To do this explicitly we follow \cite[Section 4]{leijenhorst2023solving}. 

This symmetry reduction involves the irreducible, unitary representations of $S_3$ and $S_4$. Although the irreducible, unitary representations we use involve irrational numbers, the irrationalities cancel in the final formulation, and the semidefinite program we obtain is rational whenever $\cos\theta$ is rational.

\section{Applications to spherical codes}
\label{sec:applications}

\subsection{Optimality and uniqueness of the $D_4$ root system}

In this section, we prove that the $D_4$ root system is the unique optimal kissing configuration in dimension four and is an optimal spherical code.

For this we first compute a numerically optimal solution to \eqref{pr:semialgebraic} with $n=4$ and $\theta=\pi/3$. To get a sharp bound, we use $d_1 = 14$ and $d_2 = \delta = 16$ for the truncation of the inverse Fourier transform and the sums-of-squares degrees. The resulting semidefinite program is large, and to solve it the use of the semidefinite programming solver from \cite{leijenhorst2023solving} is essential. This solver supports arbitrary precision floating-point arithmetic and exploits the low-rank structure of the constraint matrices arising from enforcing the polynomial constraints through sampling at a unisolvent set \cite{lofberg_coefficients_2004}. We compute the optimal solution to $40$ digits of precision using $256$-bit floating-point arithmetic. This takes about two weeks on $8$ cores of a modern computer equipped with $128$GB of working memory. 

The next step is to round the numerical solution to an exact optimal solution. 
Since the dimension of the optimal face is lower than the dimension of the space given by the affine constraints, simply projecting the numerically optimal solution into the affine space does not work: the resulting matrix variables will generally not be positive semidefinite. Instead, we use the recently developed rounding heuristic from \cite{cohn2024optimality}. This method gives a major speedup over the rounding heuristic developed in \cite{MR4263438}, which is crucial for the size of the semidefinite program we consider here. 

Although a semidefinite program defined over the rationals does not necessarily admit a rational optimal solution (see, e.g., \cite{MR2546336}), this is the case here, and the rounding procedure finds a rational optimal solution within $4$ hours. 
This gives an exact feasible solution $K$ with objective value $K(\emptyset, \emptyset) = 24$. 

To verify that the exact solution is indeed feasible we check that the affine constraints hold, which can be done in rational arithmetic, and we check that the solution matrices are positive semidefinite. To be able to check positive semidefiniteness, the rounding procedure writes each solution matrix in the form $B X B^{\sf T}$, where $B$ is a rectangular, rational matrix, and $X$ is a positive definite, rational matrix. We check that $X$ is indeed positive definite by computing the Cholesky decomposition in rigorous ball arithmetic. 
As part of the verification procedure, the zonal matrices $Z_\lambda$ need to be constructed,
which takes less than two days on a modern computer. The remainder of the verification procedure takes less than two hours. 

Using Sturm sequences we verify that the polynomial $p_2$ corresponding to $A_2K|_{I_{=2}}$ has roots $-1$, $ \pm 1/2$, and $0$ in the interval $[-1,1/2]$. That is, for distinct $x, y \in S^3$ with $\langle x, y\rangle \leq 1/2$, $A_2K(\{x,y\}) = 0$ if and only if $\langle x, y\rangle \in \{-1, \pm 1/2, 0\}$. In the remainder of this section, we use this fact to show the $D_4$ root system is an optimal spherical code and is the unique optimal kissing configuration up to isometry.

The script and data files to perform this verification procedure are available at \cite{data2024}. There we also make available the implementation we used for generating the proofs. Our scripts are written in Julia \cite{bezanson17} and use the Nemo computer algebra system \cite{fieker17}.

\begin{lemma}\label{lem:innerproductsincode}
    If $C \subseteq S^3$ is a subset of size $24$ with  minimal angle at least $\pi/3$, then
    \[
        \langle x, y\rangle \in \{ -1, -1/2 , 0, 1/2\}
    \]
    for all distinct $x, y \in C$.
\end{lemma}

\begin{proof}
    Let $K$ be the exact solution discussed above. By positivity of $K$ and by the linear constraints
    \[
    A_2 K(Q) \leq -1_{\Ii_{=1}}(Q),\qquad Q \in \Ii_4\setminus\{\emptyset\},
    \]
    we have
    \[
        0 \leq \sum_{\substack{J_1,J_2 \in \mathcal{I}_2\\J_1,J_2 \subseteq C}} K(J_1,J_2) = \sum_{\substack{Q \in \mathcal{I}_4\\ Q \subseteq C}} A_2K(Q) \leq K(\emptyset, \emptyset) - |C|.
    \]
    Since $|C| = K(\emptyset, \emptyset) = 24$,  equality holds throughout, so in particular, $A_2K(Q) = 0$ for all $Q \subseteq C$ with $|Q| \leq 4$. As mentioned above, for distinct $x, y \in S^3$, we have $A_2K(\{x,y\}) = 0$ if and only if $\langle x, y \rangle \in \{-1, \pm 1/2, 0\}$, which proves the lemma.
\end{proof}

This shows the $D_4$ root system corresponds to an optimal spherical code: among the $24$-point subsets of $S^3$, the minimal distance between distinct points is as large as possible.

\begin{theorem}
The $D_4$ root system is an optimal spherical code.
\end{theorem}

\begin{proof}
If there were a spherical code $C$ of cardinality $24$ with smallest angle strictly larger than $\pi/3$, then any small enough perturbation of $C$ would correspond to a kissing configuration of size $24$, which contradicts with Lemma~\ref{lem:innerproductsincode}.
\end{proof}
Note that for cases where the Bachoc-Vallentin three-point bound is sharp, optimality of the corresponding spherical code follows immediately, because in that case sharpness directly implies that there are only finitely many possible inner products; see \cite{MR2546336}. We do not know whether the same is always true for a truncation of the Lasserre hierarchy, since it is not clear whether the polynomial $p_2$ can be identically zero when the bound is sharp.

\begin{theorem}\label{thm:unique}
    The $D_4$ root system is the unique optimal kissing configuration in $\R^4$ up to isometry.
\end{theorem}

\begin{proof}
    Let $C \subseteq S^3$ be an optimal kissing configuration in $\R^4$.
    We first verify that $C$ is a root system.
    \begin{enumerate}
        \item  The vectors in $C$ must span $\R^4$, since otherwise $C$ would give a kissing configuration in $\R^3$ of size $24$.
    \item Since $C$ is a subset of the unit sphere, the only scalar multiples of $\alpha \in C$ can be $\alpha$ and $-\alpha$.
    \item Let $\alpha, \beta \in C$ and consider the reflection 
    $
    \beta' = \beta - 2 \langle \alpha, \beta \rangle \alpha
    $
    of $\beta$ through the hyperplane orthogonal to $\alpha$. By Lemma~\ref{lem:innerproductsincode}, it follows that   
    \[
    \langle \beta', \gamma\rangle \in \{\pm 1, \pm 1/2, 0 \}
    \]
    for every $\gamma \in C$. 
    So, $\beta'$ must be in $C$ by optimality of $C$.
    \item By Lemma~\ref{lem:innerproductsincode}, for $\alpha, \beta \in C$, the value $2 \langle \alpha, \beta \rangle$ is an integer. In other words, the reflection of $\beta$ through the hyperplane orthogonal to $\alpha$ is obtained by subtracting an integer multiple of $\alpha$ from $\beta$.
    \end{enumerate}
    Hence, the set $C$ is a root system in $\R^4$.

    The irreducible root systems have been classified, and the only irreducible root systems where all vectors have the same length are $A_j$, $D_j$, $E_6$, $E_7$, and $E_8$; see, for instance, \cite[Table 4.1]{sloane81}. Since all roots in $C$ have the same length, it must be a direct sum of these irreducible root systems. In other words, 
    \[
    C = \bigoplus^k_{i=1} \Phi_i
    \]
    for some $k$ and root systems $\Phi_1,\dots,\Phi_k$, where each $\Phi_k$ is isomorphic to  $A_j$, $D_j$, $E_6$, $E_7$, or $E_8$.
    
    Let us assume that $D_4$ does not occur in the decomposition. 
    By considering the dimensions, the summands must be isomorphic to $A_j$ with $1 \leq j \leq 4$ and $D_j$ with $1 \leq j \leq 3$.
    We denote by $r$ the total number of roots occurring in $C$, by $r_i$ the number of roots of $\Phi_i$, and by $d_i$ the rank of $\Phi_i$.
    For $A_j$ we have $r_j/d_j = j + 1$ and for $D_j$ we have $r_j/d_j = 2(j - 1)$.
    Hence, we have $r_i/d_i < 6$ for the root systems which occur in the decomposition.
    Furthermore, since the span of $C$ is $\R^4$, we have $\sum_{i = 1} d_i = 4$.
    We then have
    \[
        r = \sum_{i=1}^k r_i = \sum_{i=1}^k \frac{r_i}{d_i} d_i < 6 \sum_{i = 1} d_i = 24.
    \]
    Since the number of roots in $C$ is equal to $24$, this gives a contradiction.
    Hence, $C$ is $D_4$ up to orthogonal transformations.
\end{proof}

We sketch an alternative proof of Theorem~\ref{thm:unique} which does not rely on the classification of irreducible root systems. Consider an optimal spherical code. As argued in the above proof, the code is antipodal. Now \cite[Proposition~3.12]{calderbank1997Z4} implies that the corresponding set of $12$ lines is the union of $3$ orthonormal bases, with lines in different bases not orthogonal. Choose one of these bases to define coordinates. The basis elements and their negatives give $8$ points. By Lemma~\ref{lem:innerproductsincode}, the remaining 16 must have $\pm 1/2$ in every coordinate. The only possibility is to have every such point, and so the resulting configuration is unique. 

\subsection{New bound in dimension six}

We also computed the second level of the hierarchy with the parameters $d_1=14$ and $d_2 = \delta=16$ for the kissing number problem in dimensions $5$, $6$, $7$, $10$, $12$, and $16$. In dimension $6$ this gives $k(6) \leq 77$, which improves on the previously best-known upper bound of $k(6) \leq 78$ obtained using the three-point bound \cite{bachoc08}. As mentioned in the introduction, the degrees for which we perform computations are not yet high enough to get improvements in the other dimensions.

In dimension $6$, this does not give a sharp bound, so that the optimal objective value and optimal solution potentially require high algebraic degree or bit size. This means the rounding procedure from \cite{cohn2024optimality} may not be able to find an exact feasible solution here. Therefore, we solve the problem as a feasibility problem, where we add the constraint that the objective $K(\emptyset, \emptyset)$ is equal to $77.85$. Since this is strictly larger than the numerically computed optimal objective, the solver will return a strictly feasible solution (a feasible solution where all matrix variables are positive definite), from which it is easy to extract an exact feasible solution. This gives a rigorous proof of $k(6) \leq 77$. Again, the verification script and data files are available at \cite{data2024}.

\section*{Acknowledgements}

We are grateful to Henry Cohn and Fernando Oliveira for helpful discussions. We also thank  Oliveira for letting us use his unpublished Julia code for invariant sum-of-squares polynomials, and we thank Cohn for making us aware of the alternative proof below Theorem~\ref{thm:unique}. We thank Joffrey Wallaart for ICT support for the computational part of this project.

\providecommand{\bysame}{\leavevmode\hbox to3em{\hrulefill}\thinspace}
\providecommand{\MR}{\relax\ifhmode\unskip\space\fi MR }
\providecommand{\MRhref}[2]{%
  \href{http://www.ams.org/mathscinet-getitem?mr=#1}{#2}
}
\providecommand{\href}[2]{#2}

\end{document}